\documentclass[11pt,reqno]{amsart}
\usepackage{graphicx}
\usepackage{amssymb}
\usepackage[usenames,dvipsnames]{color}
\newtheorem{theorem}{Theorem}[section]
\newtheorem{lemma}{Lemma}[section]
\newtheorem{proposition}{Proposition}[section]
\newtheorem{definition}{Definition}[section]

\newtheorem{remark}{Remark}[section]

\numberwithin{equation}{section}

\def\essinf{\mathop{\rm ess~inf}}

\def\R{\mathbb{R}}

\def\Z{\mathbb{Z}}

\def\pa{\partial}

\begin{document}

\title[Transition fronts of KPP-type lattice random equations]{Transition fronts of KPP-type lattice random equations}

\author{Feng Cao}
\address{Department of Mathematics, Nanjing University of Aeronautics and Astronautics, Nanjing, Jiangsu 210016, P. R. China}
\email{fcao@nuaa.edu.cn}
\thanks{Research of F. Cao was supported by NSF of China No. 11871273, and the Fundamental Research Funds for the Central Universities No. NS2018047.}

\author{Lu Gao}
\address{Department of Mathematics, Nanjing University of Aeronautics and Astronautics, Nanjing, Jiangsu 210016, P. R. China}
\email{gaolunuaa@sina.com}





\begin{abstract}
In this paper, we investigate the existence and stability of random transition fronts of KPP-type lattice equations in random media, and explore the influence of the media and randomness on the wave profiles and wave speeds of such solutions. We first establish comparison principle for sub-solutions and super-solutions of KPP type lattice random equations and  prove the stability of positive constant equilibrium solution. Next, by constructing appropriate  sub-solutions and super-solutions, we show
the existence of random transition fronts. Finally, we prove the stability of random transition fronts of KPP-type lattice random equations.

\end{abstract}

\maketitle



\section{Introduction}

The current paper is to explore the existence and stability of  transition fronts for the following KPP-type lattice random equations
\begin{equation}\label{main-eqn}
\dot{u}_i(t) =u_{i+1}( t ) -2u_i( t) +u_{i-1}( t) +a(\theta _t\omega) u_i( t)( 1-u_i( t)) ,\ \ i\in \mathbb{Z},
\end{equation}
where $\omega\in\Omega$, $(\Omega,\mathcal{F},\mathbb{P})$ is a given probability space, $\theta_t$ is an ergodic metric dynamical system on $\Omega$, $a:\Omega\rightarrow(0,\infty)$ is measurable, and
$a^\omega(t):=a(\theta_t\omega)$ is locally H\"older continuous in $t\in\R$ for every $\omega\in\Omega$.

Equation \eqref{main-eqn} is used to model the  population dynamics of species living in patchy environments in biology and ecology (see, for example, \cite{ShKa97, ShSw90}).
It is a spatial-discrete counterpart  of the following reaction diffusion equation
\begin{equation}
\label{continuous-main-eqn}
\partial_tu=u_{xx}+a(\theta _t\omega)u(1-u),\ \ x\in\mathbb{R}.
\end{equation}

Equation \eqref{continuous-main-eqn} is widely used to model the population dynamics of species when the movement
or internal dispersal of the organisms occurs between adjacent locations randomly in spatially
continuous media.  The study of traveling wave solutions of \eqref{continuous-main-eqn} traces back to Fisher \cite{Fish37} and Kolmogorov, Petrovsky and Piskunov \cite{KPP37}
 in the special case $a(\theta _t\omega)\equiv 1$.   They investigated the existence of traveling wave solutions, that is, solutions of the form $u(x,t)=\phi(x-ct)$ with $\phi(-\infty)=1$, $\phi(+\infty)=0$. Fisher in \cite{Fish37} proved that \eqref{continuous-main-eqn} with $a(\theta _t\omega)\equiv 1$ admits traveling wave solutions if the wave speed  $c\geq2$ and showed that there are no such traveling wave solutions of slower speed.  Kolmogorov, Petrovsky, and Piskunov in \cite{KPP37} proved that for any nonnegative solution $u(x,t)$ of \eqref{continuous-main-eqn} with $a(\theta _t\omega)\equiv 1$, if at time $t=0$, $u$ is $1$ near $-\infty$ and $0$ near $\infty$, then $\lim_{t\to \infty}u(t,ct)$ is $0$ if $c>2$ and $1$ if $c<2$. $c_*:=2$ is therefore the minimal wave speed and is also called the spreading speed of \eqref{continuous-main-eqn} with $a(\theta _t\omega)\equiv 1$. The spreading properties was extended to more general monostable nonlinearities by Aronson and Weinberger \cite{ArWe78}.

Since then, traveling wave solutions of Fisher or KPP type evolution equations in spatially and temporally homogeneous media or spatially and/or temporally periodic media have been widely invetigated. The reader is referred to \cite{ArWe75, ArWe78,   BeHaNa1, BeHaNa2, BeHaRo, BeNa12, FrGa, HuSh09, Kam76, KoSh, LiZh1, LiZh2, Na09,  NoRuXi, NoXi, NRRZ12, Sa76, Sh10, Sh04,  Uch78, We02} for the study of Fisher or KPP type reaction diffusion equations in homogeneous or periodic media. As for the study of Fisher or KPP type lattice equations  in homogeneous or periodic media, the reader is referred to  \cite{ChFuGu06, ChGu02, ChGu03, HuZi94, MaZh, WuZo97, ZiHaHu93} for the existence and stability of traveling wave solutions in homogeneous media,  and to \cite{GuHa06, GuWu09, HuZi94} for the existence and stability of periodic traveling wave solutions in spatially periodic media. Recently, Cao and Shen \cite{CaSh18}  proved the existence and stability of periodic traveling wave solutions for Fisher or KPP type lattice equations in spatially and temporally periodic media.

The study of  traveling wave solutions of general time and/or space dependent Fisher or KPP type equations is attracting more and more attention due to the
presence of general time and space variations in real world problems. To study
the front propagation dynamics of Fisher or KPP type equations with general time and/or space
dependence, one first needs to properly extend the notion of traveling wave solutions in the
classical sense. Some general extension has been introduced in literature.
For example, in \cite{Sh04, Sh11}, notions of random traveling wave solutions and generalized traveling wave solutions are introduced for random Fisher or KPP type equations and
quite general time dependent Fisher or KPP type equations, respectively. In \cite{BeHa07, BeHa12}, a notion of generalized transition waves is introduced for Fisher or KPP type equations with general space and time dependence. Among others, the authors of \cite{NaRo12,NaRo15,NaRo17} proved the existence of generalized transition waves of general time dependent and space periodic, or time independent and space almost periodic Fisher or KPP type reaction diffusion equations. Zlatos \cite{Zl12} established the existence of generalized transition waves of spatially inhomogeneous Fisher or KPP type reaction diffusion equations  under some specific hypotheses. Shen \cite{Sh17} proved the stability of generalized transition waves of Fisher or KPP type reaction diffusion equations with quite general time and space dependence.

However, there is little study on the traveling wave solutions of Fisher or KPP type lattice
equations with general time and/or space dependence. Since in nature, many systems are subject to irregular influences arisen from various kind of noise, it is also of great importance to study
traveling wave solutions in random media. The purpose of our current paper is to investigate
the existence and stability of traveling wave solutions for KPP-type lattice
equations in random media under very general assumption (See {\bf (H)} below), and to understand the influence of the media and randomness on the wave profiles and wave speeds of such solutions. We note that the work \cite{Sh18} studied the existence and stability of random transition fronts for random KPP-type reaction diffusion equations.

It should be pointed out that Cao and Shen \cite{CaSh17, CaSh18} investigated the existence and stability of transition fronts for KPP-type lattice equations with general time dependence under some more restrictive assumptions. For KPP-type lattice equations in random media, although it's easy to get that the wave speed is stationary ergodic in $t$, but it is far from being obvious that the same is true for the random profile. Besides, when dealing with spatial-discrete equations, we need find another approach to get the existence of traveling wave solutions due to the lack of space regularity.

First we give some notations and assumption related to (\ref{main-eqn}). Let
$$
\underline{a}( \omega) =\underset{t-s\rightarrow \infty}{\lim\inf}\frac{1}{t-s}\int_s^t{a( \theta _{\tau}\omega)}d\tau :=\underset{r\rightarrow \infty}{\lim}\underset{t-s\ge r}{\inf}\frac{1}{t-s}\int_s^t{a( \theta _{\tau}\omega)}d\tau
$$\\
and
$$
\overline{a}(\omega) =\underset{t-s\rightarrow \infty}{\lim\sup}\frac{1}{t-s}\int_s^t{a( \theta _{\tau}\omega)}d\tau :=\underset{r\rightarrow \infty}{\lim}\underset{t-s\ge r}{\sup}\frac{1}{t-s}\int_s^t{a( \theta _{\tau}\omega)}d\tau.
$$\\

We call $\underline{a}(\cdot)$ and $\overline{a}(\cdot)$ the least mean and the greatest mean of $a(\cdot)$, respectively. It's easy to get that \\
$$ \underline{a}(\theta _t\omega ) =\underline{a}(\omega) \ \ \text{and } \ \ \overline{a}( \theta _t\omega) =\overline{a}( \omega) \ \ \mbox{ for all } t\in\mathbb{R},  $$
and
$$
 \underline{a}(\omega) =\underset{t,s\in \mathbb{Q},t-s\rightarrow \infty}{\lim\inf}\frac{1}{t-s}\int_s^t{a( \theta _{\tau}\omega)}d\tau \ \ \ \text{and\\\ } \ \ \overline{a}(\omega) =\underset{t,s\in \mathbb{Q},t-s\rightarrow \infty}{\lim\sup}\frac{1}{t-s}\int_s^t{a( \theta _{\tau}\omega)}d\tau .
$$\\
Then  $\underline{a}(\omega)$  and $\overline{a}(\omega)$  are measurable in $\omega$.

Throughout the paper, we assume that

\medskip
\noindent{\bf (H)}  $ \ 0<\underline{a}( \omega ) \le \overline{a}( \omega) <\infty$ \ for\ a.e.\ $\omega \in \Omega . $

\medskip

This implies that $\underline{a}(\cdot), a(\cdot), \overline{a}(\cdot)\in L^1(\Omega,\mathcal{F},\mathbb{P})$ (see Lemma \ref{L-lemma}). Also {\bf (H)} together with the ergodicity of the metric dynamical system $(\Omega,\mathcal{F},\mathbb{P},\{\theta_t\}_{t\in\mathbb{R}})$ imply that, there are $\underline{a},\overline{a}\in\mathbb{R}^+$ and a measurable subset $\Omega_0\subset\Omega$ with  $\mathbb{P}(\Omega_0)=1$ such that \\
$$\left\{ \begin{array}{l} 	\theta _t\Omega _0=\Omega _0\ \ \forall t\in \mathbb{R}\\ 	\underset{t-s\rightarrow \infty}{\lim\inf}\frac{1}{t-s}\int_s^t{a( \theta _{\tau}\omega)}d\tau =\underline{a}\ \ \forall \omega \in \Omega _0\\ 	\underset{t-s\rightarrow \infty}{\lim\sup}\frac{1}{t-s}\int_s^t{a( \theta _{\tau}\omega)}d\tau =\overline{a}\ \ \forall \omega \in \Omega _0.\\ \end{array} \right. $$

Let
$$l^\infty(\Z)=\{u=\{u_i\}_{i\in \Z}:\sup \limits_{i \in\Z}|u_i|<\infty\}$$
with norm $\|u\|=\|u\|_\infty=\sup_{i\in\Z}|u_i|$. Since $a(\theta_t\omega)$ is locally H\"older continuous in $t\in\R$ for every $\omega\in\Omega$, for any given $u^0\in l^\infty(\Z)$ , \eqref{main-eqn} has a unique (local) solution
$u(t;u^0,\omega)=\{u_i(t;u^0,\omega)\}_{i\in\Z}$ with $u(0;u^0,\omega)=u^0$. Note that, if $u^0_i\geq0$
for all $i\in \Z$, then $u(t;u^0,\omega)=\{u_i(t;u^0,\omega)\}_{i\in\Z}$ exists for all $t\geq0$ and $u_i(t;u^0,\omega)\geq0$ for all $i\in\Z$ and $t\geq 0$ (see Proposition \ref{comparison}).

A solution $u(t;\omega)=\{u_i(t;\omega)\}_{i\in\mathbb{Z}}$ of (\ref{main-eqn}) is called an \textit{entire solution} if it is a solution of (\ref{main-eqn}) for $t\in\mathbb{R}$.

\begin{definition} [Transition front] An entire solution $u(t;\omega)=\{u_i(t;\omega)\}_{i\in\mathbb{Z}}$ is called a  \textit{random generalized traveling wave} or a \textit{random transition front} of (\ref{main-eqn}) connecting $1$ and $0$
if for  a.e. $\omega\in\Omega$,
$$
 u_i( t;\omega ) ={\varPhi} ( i-\int_0^t{c( s;\omega) ds},\theta _t\omega)
 $$
for some $ {\varPhi}( x,\omega)$ $(x\in\R)$ and $c( t;\omega) $, where
$ {\varPhi}( x,\omega)$ and $c( t;\omega) $ are measurable in $\omega$,
and for a.e. $\omega\in\Omega$,
\begin{equation*}
0<{\varPhi}( x,\omega) <1,\ \text{and\ }\underset{x\rightarrow -\infty}{\lim}{\varPhi}( x,\theta _t\omega) =1,\ \underset{x\rightarrow \infty}{\lim}{\varPhi}( x,\theta _t\omega) =0\ \text{uniformly\ in\ }t\in \mathbb{R}.
\end{equation*}
\end{definition}

Suppose that $u(t;\omega)=\{u_i(t;\omega)\}_{i\in\mathbb{Z}}$ with $u_i( t;\omega) ={\varPhi}( i-\int_0^t{c( s;\omega) ds},\theta _t\omega )$ is a \textit{random transition front} of (\ref{main-eqn}). If $ {\varPhi} (x,\omega)$ is non-increasing in $x$ for a.e. $\omega\in\Omega$ and all $x\in\mathbb{R}$, then $u(t;\omega)$ is said to be a \textit{monotone random transition front}. If there is $\overline{c}_{\inf}\in\mathbb{R}$ such that for a.e. $\omega\in\Omega$,
$$ \underset{t-s\rightarrow \infty}{\lim\inf}\frac{1}{t-s}\int_s^t{c( \tau ;\omega ) d\tau}=\overline{c}_{\inf}, $$
then $\overline{c}_{\inf}$ is called its \textit{least mean speed}.

For given $\mu>0$, let
$$
 c_0:=\inf \limits_{\mu>0}\frac{e^{\mu}+e^{-\mu}-2+\underline{a}}{\mu}.
 $$
 By \cite[Lemma 5.1]{CaSh17},
 there is a unique $\mu^*>0$ such that
 $$ c_0=\frac{e^{\mu ^*}+e^{-\mu ^*}-2+\underline{a}}{\mu ^*} $$
  and for any $\gamma>c_0$, the equation $ \gamma=\frac{e^{\mu }+e^{-\mu}-2+\underline{a}}{\mu} $ has exactly two positive solutions for $\mu$.

Now we are in a position to state the main results on the existence and stability of random transition fronts of KPP-type lattice random equations.

\begin{theorem}\label{exist-thm}

For any given $\gamma>c_0$, there is a monotone random transition front of  \eqref{main-eqn} with least mean speed $\overline{c}_{\inf}=\gamma$. More precisely,  for any given $\gamma>c_0$,  let $0<\mu<{\mu}^*$ be such that $\frac{e^{\mu}+e^{-\mu}-2+\underline{a}}{\mu}=\gamma$. Then \eqref{main-eqn} has  a monotone random transition front $u(t;\omega)=\{u_i(t;\omega)\}_{i\in\mathbb{Z}}$ with $u_i( t;\omega ) ={\varPhi} ( i-\int_0^t{c(s;\omega,\mu ) ds},\theta _t\omega )$, where $c(t;\omega, \mu)=\frac{e^{\mu}+e^{-\mu}-2+a(\theta_t\omega)}{\mu}$ and hence $\overline{c}_{\inf}=\frac{e^{\mu}+e^{-\mu}-2+\underline{a}}{\mu}=\gamma$.
Moreover, for any $\omega\in \Omega_0$,
$$
\lim\limits_{x\rightarrow -\infty}{\varPhi}(x,\theta _t\omega)=1 \text{ and }
\lim\limits_{x\rightarrow \infty}\frac{{\varPhi}( x,\theta _t\omega)}{e^{-\mu x}}=1\ \text{uniformly\ in\ }t\in \mathbb{R}.$$\\
\end{theorem}

\begin{remark}
$(1)$ Let
$$
c_*( \omega ) =\sup\{ c:\limsup_{t\to\infty}\sup_{s\in\R, i\in\Z, |i|\le ct}| u_i( t;u^0,\theta _s\omega ) -1 |=0\ \ \text{for\ all\ }u^0\in l_{0}^{\infty}( \mathbb{Z} ) \},
$$ where
$$
l_{0}^{\infty}( \mathbb{Z} ) =\{ u=\{ u_i \} _{i\in \mathbb{Z}}\in l^{\infty}( \mathbb{Z} ) \ :\ u_i\ge 0\ \text{for\ all\ }i\in \mathbb{Z},\ u_i=0\ \text{for\ }| i |\gg 1,\ \{ u_i \} \ne 0 \} .
$$ Then by the similar arguments as proving \cite[Theorem 1.3 (2)]{CaSh17}, we can get that for a.e. $\omega\in\Omega$, $c_*( \omega )=c_0$. If $u(t;\omega)=\{u_i(t;\omega)\}_{i\in\mathbb{Z}}$ with $u_i( t;\omega ) ={\varPhi} ( i-\int_0^t{c( s;\omega) ds},\theta _t\omega)$ is a random transition front of \eqref{main-eqn} connecting $1$ and $0$, then   
$\underset{x\le z}{\inf}\ \underset{s\in \mathbb{R}}{\inf}\ \varPhi ( x,\theta _s\omega ) >0$ for all $z\in \mathbb{R}$.
Therefore, we can choose $u^0_\omega\in l_{0}^{\infty}(\mathbb{Z})$ such that
$  u^0_\omega\leq\Phi(x,\theta_s\omega)$ for all $s\in\R$. Let $0<\epsilon\ll 1$.
Then by $c_*( \omega )=c_0$ and the comparison principle, we have that
\begin{align*}
1&=\underset{t\rightarrow \infty}{\lim\inf}\underset{s\in \mathbb{R}}{\inf}u_{[ ( c_0-\epsilon ) t ]}( t;u^0_\omega,\theta _s\omega )\\
&\le \underset{t\rightarrow \infty}{\lim\inf}\underset{s\in \mathbb{R}}{\inf}u_{[ ( c_0-\epsilon ) t ]}( t;\varPhi ( \cdot ,\theta _s\omega ) ,\theta _s\omega )\\
&=\underset{t\rightarrow \infty}{\lim\inf}\underset{s\in \mathbb{R}}{\inf}\varPhi ( [ ( c_0-\epsilon ) t ] -\int_0^t{c( \tau ;\theta _s\omega ) d\tau ,\theta _{t+s}\omega} ).
\end{align*}
Note that
$
\int_0^{t+s}{c( \tau ;\omega ) d\tau}=\int_0^s{c( \tau ;\omega ) d\tau}+\int_0^t{c( \tau ;\theta _s\omega ) d\tau}.
$ Then there is a constant $M(\omega)$ such that
$
  ( c_0-\epsilon ) t \leq\int_0^{t+s}{c( \tau ;\omega ) d\tau}-\int_0^s{c( \tau ;\omega ) d\tau}+M(\omega)$ for all $t>0$,  $s\in\R$. Hence,
$$\overline{c}_{\inf}=
\underset{t\rightarrow \infty}{\lim\inf}\underset{s\in \mathbb{R}}{\inf}\frac{\int_0^{t+s}{c( \tau ;\omega ) d\tau}-\int_0^s{c( \tau ;\omega ) d\tau}}{t}\geq c_0-\epsilon.
$$
By the arbitrariness of $\epsilon>0$, we get $\overline{c}_{\inf}\geq c_0$. This implies that there is no random transition front of \eqref{main-eqn} with least mean speed less than $c_0$.

$(2)$ As for the critical random transition front of \eqref{main-eqn}, that is, random transition front of \eqref{main-eqn} with least mean speed  $\overline{c}_{\inf}=c_0$. The approach used in \cite{CaSh17} can't be applied as the stationary ergodic property of the critical random profile can't be guaranteed. We leave this question open.

\end{remark}

\begin{theorem}\label{stable-thm}
For given $\mu\in(0,\mu^*)$, the random transition front
$u(t;\omega)=\{u_i(t;\omega)\}_{i\in\mathbb{Z}}$, $u_i( t;\omega ) ={\varPhi} ( i-\int_0^t{c(s;\omega,\mu ) ds},\theta _t\omega )$
with
$ \underset{i\rightarrow \infty}{\lim}\frac{u_i(t,\omega)}{e^{-\mu ( i-\int_0^t{c(s;\omega ,\mu ) ds})}}=1$
$(c( t ;\omega ,\mu) =\frac{e^{\mu}+e^{-\mu}-2+a\left( \theta _{t}\omega \right)}{\mu} )$
is asymptotically stable, that is, for any $ \omega \in \Omega _0 $ and
$ u^0\in l^{\infty}( \mathbb{Z})$
satisfying that
\begin{equation*}
\label{stable1}
\underset{i\le i_0}{\inf}u_{i}^{0}>0\ \ \ \forall i_0\in \mathbb{Z},\ \ \underset{i\rightarrow \infty}{\lim}\frac{u_{i}^{0}}{u_i(0;\omega)}=1,
\end{equation*}
there holds
 $$ \underset{t\rightarrow \infty}{\lim}\| \frac{u_{\cdot}(t;u^0,\omega )}{u_{\cdot}(t, \omega)}-1\| _{l^{\infty}}=0.$$
\end{theorem}

The rest of the paper is organized as follows. In Section 2, we establish the comparison principle for sub-solutions and super-solutions of KPP-type lattice random equations \eqref{main-eqn} and stability of the positive constant equilibrium solution. Also, we give in Section 2 some results including the technical lemmas for the use in later section. We investigate the existence and stability of random traveling waves for KPP-type lattice equations in random media and prove Theorem \ref{exist-thm} and \ref{stable-thm} in Section 3.

\section{Preliminary}

In this section, we present some preliminary materials to be used in later sections. We first present a comparison principle for sub-solutions and super-solutions of \eqref{main-eqn}. Then we prove the stability of the positive constant equilibrium solution $u=1$ and the convergence of solutions on compact subsets. Finally we present some technical lemmas.

Consider now the following space continuous version of \eqref{main-eqn},
\begin{equation}\label{main-eqn2}
\pa_tv(x,t)=Hv(x,t)+a(\theta_t\omega)v(x,t) (1-v(x,t)), \quad\quad x\in \R,\, t\in\R,\ \omega\in\Omega,
\end{equation}
where  $$Hv(x,t)=v(x+1,t)+v(x-1,t)-2v(x,t), \quad x\in \R,\, t\in\R.$$
Recall
$$
l^\infty(\Z)=\{u:\Z\to\R\,:\, \sup_{x\in\Z}|u(x)|<\infty\}.
$$
Let
$$
l^\infty(\R)=\{u:\R\to\R\,:\, \sup_{x\in\R}|u(x)|<\infty\}
$$
with norm $\|u\|=\sup_{x\in\R}|u(x)|$.
Let  $$ l^{\infty,+}(\Z)=\{u\in l^{\infty}(\Z):\inf_{i\in\Z}u_i\ge 0\},\quad l^{\infty,+}(\R)=\{u\in l^\infty(\R)\,:\, \inf_{x\in\R}u(x)\ge 0\}$$

For any $u_0\in l^\infty(\R)$, let $u(x,t;u_0,\omega)$ be the solution of \eqref{main-eqn2} with $u(x,0;u_0,\omega)=u_0(x)$. Recall that
for any $u^0\in l^\infty(\Z)$, $u(t;u^0,\omega)=\{u_i(t;u^0,\omega)\}_{i\in\Z}$ is the solution of \eqref{main-eqn} with $u_i(0;u^0,\omega)=u^0_i$ for $i\in\Z$.

A function $v(x,t;\omega)$ on $\R\times[0,T)$ which is continuous in $t$
is called a {\it super-solution} or {\it sub-solution} of \eqref{main-eqn2} (resp. \eqref{main-eqn}) if for a.e. $\omega\in\Omega$ and any given $x\in\R$ (resp. $x\in\Z$), $v(x,t;\omega)$ is  absolutely continuous
in  $t\in [0,T)$,  and
$$ v_t(x,t;\omega)\ge Hv(x,t;\omega)+a(\theta_t\omega)v(x,t;\omega) (1-v(x,t;\omega))\quad {\rm for}\quad  \,\, t\in [0,T)$$
or
$$ v_t(x,t;\omega)\le  Hv(x,t;\omega)+a(\theta_t\omega)v(x,t;\omega) (1-v(x,t;\omega))\quad {\rm for}\quad  \,\, t\in[0,T).$$

\begin{proposition}[Comparison principle]
\label{comparison}
\begin{itemize}
\item[(1)]
If  $u_1(x,t;\omega)$ and $u_2(x,t;\omega)$ are bounded sub-solution and super-solution of \eqref{main-eqn2} $($resp. \eqref{main-eqn}$)$ on $[0,T)$, respectively, and $u_1(\cdot,0;\omega)\leq u_2(\cdot,0;\omega)$, then $u_1(\cdot,t;\omega)\leq u_2(\cdot,t;\omega)$
for $t\in[0,T)$.

\item[(2)]  Suppose that $u_1(x,t;\omega)$, $u_2(x,t;\omega)$ are bounded and satisfy that for any given $x\in\R$ $($resp. $x\in\Z)$,
 $u_1(x,t;\omega)$ and $u_2(x,t;\omega)$ are absolutely continuous in $t\in[0,\infty)$, and
 \begin{align*}
 &\pa_t u_2(x,t;\omega)-(Hu_2(x,t;\omega)+a(\theta_t\omega)u_2(x,t;\omega) (1-u_2(x,t;\omega)))\\
 &>\pa_t u_1(x,t;\omega)-(Hu_1(x,t;\omega)+a(\theta_t\omega)u_1(x,t;\omega) (1-u_1(x,t;\omega)))
 \end{align*}
 \item[]for $t>0$. Moreover, suppose that $u_2(\cdot,0;\omega)\geq u_1(\cdot,0;\omega)$. Then $u_2(\cdot,t;\omega)>u_1(\cdot,t;\omega)$ for  $t>0$.

\item[(3)] If $u_0\in l^{\infty,+}(\R)$ $($resp. $u^0\in l^{\infty,+}(\Z))$, then $u(x,t;u_0,\omega)$ $($resp. $u(t;u^0,\omega))$ exists and
$u(\cdot,t;u_0,\omega)\ge 0$ $($resp. $u(t;u^0,\omega)\ge 0)$ for all $t\ge 0$.
\end{itemize}
\end{proposition}

\begin{proof} We prove the proposition for \eqref{main-eqn2}. It can be proved similarly for \eqref{main-eqn}.

\smallskip

(1) We prove (1)  by modifying the arguments of \cite[Proposition 2.4]{HuShVi08}.

Let $Q(x,t;\omega)=e^{ct}(u_2(x,t;\omega)-u_1(x,t;\omega))$, where $c:=c(\omega)$ is to be determined later. Then there is a measurable subset $\bar{\Omega}$ of $\Omega$ with $\mathbb{P}(\bar{\Omega})=0$ such that for any $\omega\in\Omega\setminus\bar\Omega$, we have
\begin{align}
\label{difference}
\partial _tQ( x,t;\omega)
=&e^{ct}( \partial _tu_2( x,t;\omega ) -\partial _tu_1( x,t;\omega)) +ce^{ct}( u_2( x,t;\omega) -u_1( x,t;\omega))\nonumber\\
\ge& e^{ct}( Hu_2( x,t;\omega) -Hu_1( x,t;\omega)  +a( \theta _t\omega ) u_2( x,t;\omega)( 1-u_2( x,t;\omega)) \nonumber\\
& -a(\theta_t\omega)u_1( x,t;\omega)(1-u_1( x,t;\omega))) +cQ(x,t;\omega) \nonumber \\
=&HQ( x,t;\omega) +e^{ct}a(\theta _t\omega)(u_2( x,t;\omega) -u_1( x,t;\omega))( 1-u_2( x,t;\omega))\nonumber\\
&  -e^{ct}a( \theta _t\omega)( u_2( x,t;\omega) -u_1( x,t;\omega )) u_1( x,t;\omega ) +cQ( x,t;\omega)\nonumber\\
=&Q( x+1,t;\omega) +Q( x-1,t;\omega) +( b( x,t;\omega ) -2+c ) Q( x,t;\omega)
\end{align}
for  $ x\in\R$ and $t\in[0,T]$, where
$$ b( x,t;\omega) =a(\theta _t\omega)( 1-u_1( x,t;\omega ) -u_2( x,t;\omega ) ) \ \ \text{for\ }x\in \mathbb{R},\ t\in [ 0,T ] . $$
Let $p(x,t;\omega)=b(x,t;\omega)-2+c$. By the boundedness of $u_1$ and $u_2$, we can choose $c=c(\omega)>0$ such that
$$\inf _{(x,t)\in\R\times [0,T]}p(x,t;\omega)>0.$$
We claim that  $Q(x,t;\omega)\geq 0$ for $x\in\R$ and $t\in[0,T]$.

Let $p_0(\omega)=\sup\limits_{(x,t)\in\R\times [0,T]}p(x,t;\omega)$. It suffices to prove the claim for $x\in\R$ and $t\in(0,T_0]$ with $T_0=\min(T,\frac{1}{p_0(\omega)+2})$. Assume that there are $\tilde x\in\R$ and $\tilde t\in(0,T_0]$ such that $Q(\tilde x,\tilde t;\omega)<0$. Then there is $t^0\in(0,T_0)$ such that
$$Q_{\inf}(\omega):=\inf_{(x,t)\in\R\times [0,t^0]}Q(x,t;\omega)<0.$$ Observe that there are $x_n\in\R$ and $t_n\in(0,t^0]$ such that
$$Q(x_n,t_n;\omega)\to Q_{\inf}(\omega)\quad\mbox{ as }\,n\to\infty.$$
By \eqref{difference}  and the fundamental theorem of calculus for Lebesgue integrals, we get
\begin{align}
Q(x_n,t_n;\omega)-Q(x_n,0;\omega)
&\geq\int^{t_n}_0[Q(x_n+1,t;\omega)+Q(x_n-1,t;\omega)+p(x_n,t;\omega)Q(x_n,t;\omega)]dt \nonumber\\
&\geq\int^{t_n}_0[2Q_{\inf}(\omega)+p(x_n,t;\omega)Q_{\inf}(\omega)]dt \nonumber\\
&\geq t^0(2+p_0(\omega))Q_{\inf}(\omega)\quad\quad\mbox{ for }\,n\geq1.\nonumber
\end{align}
Note that $Q(x_n,0;\omega)\geq0$, we then have
$$Q(x_n,t_n;\omega)\geq t^0(2+p_0(\omega))Q_{\inf}(\omega)\quad\quad\mbox{ for }\,n\geq1.$$
Letting $n\to\infty$, we obtain
$$Q_{\inf}(\omega)\geq t^0(2+p_0(\omega))Q_{\inf}(\omega)>Q_{\inf}(\omega).$$
A contradiction. Hence the claim is true and $u_1(x,t;\omega)\leq u_2(x,t;\omega)$ for $\omega\in\Omega\setminus\bar\Omega$, $x\in\R$ and $t\in[0,T]$.

(2) For $ \omega\in\Omega\setminus\bar\Omega$, by the similar arguments as getting \eqref{difference}, we can find $c(\omega)$, $\mu(\omega)>0$ such that
$$\pa_t Q(x,t;\omega)>Q(x+1,t;\omega)+Q(x-1,t;\omega)+\mu(\omega) Q(x,t;\omega)\quad\mbox{ for }\, x\in\R, \, t>s,$$
where $Q(x,t;\omega)=e^{c(\omega)t}(u_2(x,t;\omega)-u_1(x,t;\omega))$. Thus we have that for  $x\in\R$,
$$
Q(x,t;\omega)>Q(x,0;\omega)+\int_0 ^t \big(Q(x+1,\tau;\omega)+Q(x-1,\tau;\omega)+\mu(\omega) Q(x,\tau;\omega)\big)d\tau.
$$
By the arguments in (1), $Q(x,t;\omega)\ge 0$ for all $x\in\R$ and $t\ge 0$. It then follows that
$Q(x,t;\omega)>Q(x,0;\omega)\ge 0$ and hence $u_2(x,t;\omega)>u_1(x,t;\omega)$ for $\omega\in\Omega\setminus\bar\Omega$, $x\in\R$ and $t>0$.

(3) By (1), for any $u_0\in l^{\infty,+}(\R)$, $0\le u(\cdot,t;u_0,\omega)\le \max\{\|u_0\|,1\}$ for all $t>0$ in the existence interval of $u(\cdot,t;u_0,\omega)$. It then follows that $u(\cdot,t;u_0,\omega)$ exists and $u(\cdot,t;u_0,\omega)\ge 0$ for all $t\ge 0$.
\end{proof}

We have the following proposition on the stability of the constant equilibrium solution $u=1$.
\begin{proposition}
\label{stable-lemma}
For every $u_0\in l^{\infty}( \mathbb{R} )$ with $ \underset{x\in \mathbb{R}}{\inf}u_{0}(x)>0 $ and for every $ \omega \in \Omega $, we have that $$ \lVert u( x, t;u_0,\omega ) -1 \rVert _{\infty}\rightarrow 0\ \ \,as\ \ \,\, t\rightarrow \infty. $$
\end{proposition}

\begin{proof}
The proof is similar to that of \cite[Theorem 1.1]{Sh18}. We give the details for completeness.

For $ u_0\in l^{\infty}(\mathbb{R})$ with $ \underset{x\in \mathbb{R}}{\inf} u_0(x)>0 $ and $ \omega \in \varOmega $. Let $\underline{u}_0:=\text{min}\{1, \underset{x\in \mathbb{R}}{\inf}u_0(x)\}$ and $\overline{u}_0:=\text{max}\{1, \underset{x\in \mathbb{R}}{\sup} u_0(x)\}$. It follows from Proposition \ref{comparison} that
\begin{equation}
\label{stable-eqn1}
\underline{u}_0\leq u(x,t;\underline{u}_0,\omega)\leq \text{min}\{1, u(x,t;u_0,\omega)\},\ \ \ \forall x\in\mathbb{R},\  t\geq 0
\end{equation}
and
\begin{equation}
\label{stable-eqn2}
\text{max}\{1, u(x,t;u_0,\omega)\}\leq u(x,t;\overline{u}_0,\omega)\leq \overline{u}_0,\ \ \ \forall x\in\mathbb{R},\  t\geq 0.
\end{equation}
Note that  $\underline{u}_0$ and $\overline{u}_0$ are constants. Then by the uniqueness of solution of \eqref{main-eqn2} with respect to the initial value, we obtain that
$$u(x,t;\underline{u}_0,\omega)=u(0,t;\underline{u}_0,\omega)\ \ \text{and} \ \ u(x,t;\overline{u}_0,\omega)=u(0,t;\overline{u}_0,\omega)\ \ \ \forall x\in\mathbb{R},\  t\geq 0.$$
Since the functions $\underline{u}(t) =\left( \frac{1}{u( 0,t;\underline{u}_0,\omega)}-1 \right) e^{\int_0^t{a( \theta _s\omega) ds}}$ and
$\overline{u}(t) =\left( 1-\frac{1}{u( 0,t;\overline{u}_0,\omega)}\right) e^{\int_0^t{a(\theta _s\omega) ds}}$ satisfy
$$ \frac{d}{dt}\underline{u}=\frac{d}{dt}\overline{u}=0,\ \ t>0,$$
we get that
$$\underline{u}(t)=\underline{u}(0)\ \ \text{and}\ \ \overline{u}(t)=\overline{u}(0),\ \ \forall t\geq 0.$$
Thus,
\begin{equation}
\label{stable-eqn3}
1-u(x,t;\underline{u}_0,\omega)=\underline{u}(0) u(x,t;\underline{u}_0,\omega) e^{-\int_0^t{a( \theta _s\omega) ds}}
\end{equation}
and
\begin{equation}
\label{stable-eqn4}
u(x,t;\overline{u}_0,\omega)-1=\overline{u}(0) u(x,t;\overline{u}_0,\omega) e^{-\int_0^t{a( \theta _s\omega) ds}}.
\end{equation}
By \eqref{stable-eqn1} and \eqref{stable-eqn2}, we have that
$$ 0<\underline{u}_0 \leq u(x,t;\underline{u}_0,\omega)\leq u(x,t;\overline{u}_0,\omega)\leq \overline{u}_0,\ \ \ \forall x\in\mathbb{R},\  t\geq 0.$$
It then follows from \eqref{stable-eqn1}, \eqref{stable-eqn2}, \eqref{stable-eqn3} and \eqref{stable-eqn4} that
$$|u(x,t;u_0,\omega)-1| \leq \overline{u}_0\ \text{max}\{\overline{u}(0),\underline{u}(0)\}e^{-\int_0^t{a( \theta _s\omega) ds}},\ \ \ \forall x\in\mathbb{R},\  t\geq 0.$$
The Lemma thus follows.
\end{proof}

\begin{proposition}
\label{convergence-lemma}
Suppose that $u_{0n},u_0\in l^{\infty,+}(\R)$ $(n=1,2,\cdots)$ with $\{\|u_{0n}\|\}$ being  bounded.
If $u_{0n}(x)\to u_0(x)$ as $n\to\infty$
 uniformly in $x$ on bounded sets,
{ then for each $t>0$}, $u(x,t;u_{0n},\theta_{t_0}\omega)- u(x,t;u_0,\theta_{t_0}\omega)\to 0$ as $n\to\infty$ uniformly in $x$ on bounded sets and $t_0\in\R$.
\end{proposition}

\begin{proof}
It can be proved by the similar arguments in \cite[Proposition 2.2]{CaSh17}.

Fix any $\omega\in\Omega$. Let $v^n(x,t;\theta_{t_0}\omega)=u(x,t;u_{0n},\theta_{t_0}\omega)-u(x,t;u_0,\theta_{t_0}\omega)$. Then $v^n(x,t;t_0)$ satisfies
\begin{equation*}
v^n_t(x,t;\theta_{t_0}\omega)=H v^n(x,t;\theta_{t_0}\omega)+b_n(x,t;\theta_{t_0}\omega)v^n(x,t;\theta_{t_0}\omega),
\end{equation*} where $b_n(x,t;\theta_{t_0}\omega)=a(\theta_{t+t_0}\omega)(1-u(x,t;u_{0n},\theta_{t_0}\omega)-u(x,t;u_0,\theta_{t_0}\omega)) $.
Observe that $\{ b_n(x,t;\theta_{t_0}\omega) \}_n$ is uniformly bounded.

Take a $\lambda>0$. Let
$$
X(\lambda)=\{u:\R\to\R\,|\, u(\cdot) e^{-\lambda |\cdot|}\in l^\infty(\R)\}
$$
with norm $\|u\|_\lambda=\|u(\cdot)e^{-\lambda |\cdot|}\|_{l^\infty(\R)}$.
Note that $H:X(\lambda)\to X(\lambda)$ generates an analytic  semigroup,
and there are $M>0$ and $\alpha>0$ such that
$$
\|e^{H t}\|_{X(\lambda)}\leq M e^{\alpha t}\quad \forall { t\ge 0}.
$$
Hence,
\begin{align*}
v^n(\cdot,t;\theta_{t_0}\omega)=&e^{ Ht}v^n(\cdot,0;\theta_{t_0}\omega)+\int_0^t e^{ H(t-\tau)}b_n(\cdot,\tau;\theta_{t_0}\omega)v^n(\cdot,\tau;\theta_{t_0}\omega)d\tau
\end{align*}
and
then
\begin{align*}
\|v^n(\cdot,t;\theta_{t_0}\omega)\|_{X(\lambda)}
&\leq M e^{\alpha t}\|v^n(\cdot,0;\theta_{t_0}\omega)\|_{X(\lambda)}\\
&+M\sup_{t_0\in\R,\tau\in[0,t],x\in\R} |b_n(x,\tau;\theta_{t_0}\omega)|\int_0^ t e^{\alpha(t-\tau)}\|v^n(\cdot,\tau;\theta_{t_0}\omega)\|_{X(\lambda)}d\tau.
\end{align*}
By Gronwall's inequality,
$$ \lVert v^n( \cdot ,t;\theta_{t_0}\omega) \rVert _{X( \lambda)}\leq e^{( \alpha +M\sup_{t_0\in \text{R,}\tau \in [ 0,t ] ,x\in \text{R}}|b_n(x,\tau;\theta_{t_0}\omega)| ) t}( M\lVert v^n( \cdot ,0;\theta_{t_0}\omega ) \rVert _{X( \lambda )}). $$
Note that $\|v^n(\cdot,0;\theta_{t_0}\omega)\|_{X(\lambda)}\to 0$ uniformly in $t_0\in\R$. It then follows that
$$
\|v^n(\cdot,t;\theta_{t_0}\omega)\|_{X(\lambda)}\to 0\quad {\rm as}\quad n\to\infty
$$
uniformly in $t_0\in\R$ and then
$$
u(x,t;u_{0n},\theta_{t_0}\omega)- u(x,t;u_0,\theta_{t_0}\omega)\to 0 \quad {\rm as}\quad
n\to\infty
$$
uniformly in $x$ on bounded sets and $t_0\in\R$.
\end{proof}

Now we present some lemmas including the technical results.

\begin{lemma}
\label{L-lemma}
$\underline{a}(\cdot), a(\cdot), \overline{a}(\cdot)\in\L^1(\Omega, \mathcal{F}, \mathbb{P})$. Also $\underline{a}(\omega)$ and $\overline{a}(\omega)$ are independent of $\omega$ for a.e. $\omega\in\Omega$.
\end{lemma}

\begin{proof}
 It follows from \cite[Lemma 2.1]{Sh18}.
\end{proof}

\begin{lemma}\label{technical}
Suppose that for $\omega\in\Omega$, $a^\omega(t)=a(\theta_t\omega)\in C(\mathbb{R},(0,\infty))$. Then for a.e. $\omega\in\Omega$,
$$\underline{a} =\underset{A\in W_{loc}^{1,\infty}\left( \mathbb{R} \right) \cap L^{\infty}\left( \mathbb{R} \right)}{\mathrm{sup}}\underset{t \in \mathbb{R}}\essinf ( A^{\prime}+a^{\omega} )(t). $$
\end{lemma}

\begin{proof}
It follows from \cite[Lemma 2.2] {Sh18} and Lemma \ref{L-lemma}.
\end{proof}

\begin{lemma}
\label{sub-solution-lemma}
Let $\omega\in\Omega_0$. Then for any $\mu$, $\tilde{\mu}$ with $ 0<\mu <\tilde{\mu}<\min\{ 2\mu ,\mu ^* \},$
there exist $ \{  t_k \} _{k\in \mathbb{Z}} $ with $ t_k<t_{k+1} $
and $ \underset{k\rightarrow \pm \infty}{\lim}t_k=\pm \infty  $, $ A_{\omega}\in W_{loc}^{1,\infty}( \mathbb{R}) \cap L^{\infty}( \mathbb{R})$ with $ A_{\omega}(\cdot) \in C^1(( t_k,t_{k+1}))  $ for $ k\in \mathbb{Z} $, and $ d_{\omega}>0 $ such that for any $ d\ge d_{\omega} $  the function
$$ \tilde{v}^{\mu ,d,A_{\omega}}( x,t, \omega) :=e^{-\mu ( x-\int_0^t{c( s ;\omega ,\mu ) ds} )}-de^{( \frac{\tilde{\mu}}{\mu}-1 ) A_{\omega}( t ) -\tilde{\mu}( x-\int_0^t{c( s ;\omega ,\mu ) ds} )} $$
satisfies
$$ \partial _t \tilde{v}^{\mu ,d,A_{\omega}}\le H\tilde{v}^{\mu ,d,A_{\omega}}+a( \theta _t\omega ) \tilde{v}^{\mu ,d,A_{\omega}}( 1-\tilde{v}^{\mu ,d,A_{\omega}} )  $$
for $ t\in ( t_k,t_{k+1} )$, $ x\ge \int_0^t{c( s ;\omega ,\mu ) ds}+\frac{\ln d}{\tilde{\mu}-\mu}+\frac{A_{\omega}\left( t \right)}{\mu} $, $ k\in \mathbb{Z} $.
 \end{lemma}

 \begin{proof}
For given $\omega\in\Omega_0$ and $ 0<\mu <\tilde{\mu}<\min \{  2\mu, {\mu }^* \} $, by the arguments in the proof of \cite[Lemma 5.1]{CaSh17} we can get that $\frac{e^{\tilde{\mu}}+e^{-\tilde{\mu}}-2+\underline{a}}{\tilde{\mu}}<\frac{e^{\mu}+e^{-\mu}-2+\underline{a}}{\mu} $, and hence $ \underline{a} >\frac{\mu ( e^{\tilde{\mu}}+e^{-\tilde{\mu}}-2 ) -\tilde{\mu}( e^{\mu}+e^{-\mu}-2 )}{\tilde{\mu}-\mu} $. Let $ 0<\delta \ll 1 $ be such that $ ( 1-\delta )\underline{ a} >\frac{\mu ( e^{\tilde{\mu}}+e^{-\tilde{\mu}}-2 ) -\tilde{\mu}( e^{\mu}+e^{-\mu}-2 )}{\tilde{\mu}-\mu} $.
It then follows from  Lemma \ref{technical} that there exist $ T>0 $ and $ A_{\omega}\in W_{loc}^{1,\infty}( \mathbb{R} ) \cap L^{\infty}( \mathbb{R} )$ such that $ A_{\omega}( \cdot ) \in C^1(( t_k,t_{k+1}))  $ with $ t_k=kT $ for $ k\in \mathbb{Z} $, and \begin{equation}\label{inequality-A(t)}
( 1-\delta ) a( \theta _t\omega ) +A_{\omega}^{\prime}( t)\ge  \frac{\mu ( e^{\tilde{\mu}}+e^{-\tilde{\mu}}-2 ) -\tilde{\mu}( e^{\mu}+e^{-\mu}-2 )}{\tilde{\mu}-\mu}
\end{equation}
for all $ t\in ( t_k,t_{k+1})  $, $ k\in \mathbb{Z} $.

Now fix $\delta>0$ and $A_{\omega}( t )$ chosen in the above inequality. Let $ \xi ( x,t;\omega ) =x-\int_0^t{c( s ;\omega ,\mu) ds} $, and $ \tilde{v}^{\mu ,d,A_{\omega}}( x,t, \omega) :=e^{-\mu   \xi ( x,t;\omega ) }-de^{( \frac{\tilde{\mu}}{\mu}-1 ) A_{\omega}( t ) -\tilde{\mu}   \xi ( x,t;\omega  )} $ with $ d>1 $ to be determined later.
Note that $c(t;\omega, \mu)=\frac{e^{\mu}+e^{-\mu}-2+a(\theta_t\omega)}{\mu}$.
Then we have
\begin{align}
\label{more}
\partial_t& \tilde{v}^{\mu ,d,A_{\omega}}-( H\tilde{v}^{\mu ,d,A_{\omega}}+a( \theta _t\omega ) \tilde{v}^{\mu ,d,A_{\omega}}( 1-\tilde{v}^{\mu ,d,A_{\omega}} ) )\nonumber\\
=&\mu c( t;\omega,\mu ) e^{-\mu \xi ( x,t;\omega )}+d( -( \frac{\tilde{\mu}}{\mu}-1 ) A_{\omega}^{\prime}( t ) -\tilde{\mu}c( t;\omega,\mu ) ) e^{( \frac{\tilde{\mu}}{\mu}-1 ) A_{\omega}( t ) -\tilde{\mu}\xi ( x,t;\omega )}\nonumber\\
&-(( e^{\mu}+e^{-\mu}-2 ) e^{-\mu \xi ( x,t;\omega )}-d( e^{\tilde{\mu}}+e^{-\tilde{\mu}}-2 ) e^{( \frac{\tilde{\mu}}{\mu}-1 ) A_{\omega}( t ) -\tilde{\mu}\xi ( x,t;\omega)})  \nonumber\\
 &-a( \theta _t\omega ) ( e^{-\mu \xi ( x,t;\omega )}-de^{( \frac{\tilde{\mu}}{\mu}-1 ) A_{\omega}( t ) -\tilde{\mu}\xi ( x,t;\omega )} ) ( 1-( e^{-\mu \xi ( x,t;\omega )}-de^{( \frac{\tilde{\mu}}{\mu}-1 ) A_{\omega}( t ) -\tilde{\mu}\xi ( x,t;\omega )} ) )\nonumber\\
 =& d( -( \frac{\tilde{\mu}}{\mu}-1 ) A_{\omega}^{\prime}( t ) -\tilde{\mu}c( t;\omega,\mu ) + e^{\tilde{\mu}}+e^{-\tilde{\mu}}-2  +a( \theta _t\omega ) ) e^{( \frac{\tilde{\mu}}{\mu}-1 ) A_{\omega}( t ) -\tilde{\mu}\xi ( x,t;\omega )}\nonumber\\
 & +a( \theta _t\omega ) ( e^{-\mu \xi ( x,t;\omega )}-de^{( \frac{\tilde{\mu}}{\mu}-1 ) A_{\omega}( t ) -\tilde{\mu}\xi ( x,t;\omega )} ) ^2 \nonumber\\
 =&d( \frac{\tilde{\mu}}{\mu}-1 ) ( -A_{\omega}^{\prime}( t ) +\frac{\mu ( e^{\tilde{\mu}}+e^{-\tilde{\mu}}-2 ) -\tilde{\mu}( e^{\mu}+e^{-\mu}-2 )}{\tilde{\mu}-\mu}-a( \theta _t\omega ) ) e^{( \frac{\tilde{\mu}}{\mu}-1 ) A_{\omega}( t ) -\tilde{\mu}\xi ( x,t;\omega )} \nonumber\\
 &+a( \theta _t\omega ) e^{-2\mu \xi ( x,t;\omega )}-d( 2e^{-\mu \xi ( x,t;\omega )}-de^{( \frac{\tilde{\mu}}{\mu}-1 ) A_{\omega}( t) -\tilde{\mu}\xi ( x,t;\omega )} ) e^{( \frac{\tilde{\mu}}{\mu}-1 ) A_{\omega}( t ) -\tilde{\mu}\xi ( x,t;\omega )}a( \theta _t\omega )\nonumber\\
=&d( \frac{\tilde{\mu}}{\mu}-1 ) ( \frac{\mu ( e^{\tilde{\mu}}+e^{-\tilde{\mu}}-2 ) -\tilde{\mu}( e^{\mu}+e^{-\mu}-2 )}{\tilde{\mu}-\mu}-( 1-\delta ) a( \theta _t\omega ) -A_{\omega}^{\prime}( t ) )e^{( \frac{\tilde{\mu}}{\mu}-1 ) A_{\omega}( t ) -\tilde{\mu}\xi ( x,t;\omega )} \nonumber\\
&  +( e^{-( 2\mu -\tilde{\mu} ) \xi ( x,t;\omega )}-d\delta ( \frac{\tilde{\mu}}{\mu}-1 ) e^{( \frac{\tilde{\mu}}{\mu}-1 ) A_{\omega}( t )} ) a( \theta _t\omega ) e^{-\tilde{\mu}\xi ( x,t;\omega )} \nonumber\\
&+d( -2e^{-\mu \xi ( x,t;\omega )}+de^{( \frac{\tilde{\mu}}{\mu}-1 ) A_{\omega}( t ) -\tilde{\mu}\xi ( x,t;\omega )} ) e^{( \frac{\tilde{\mu}}{\mu}-1 ) A_{\omega}( t ) -\tilde{\mu}\xi ( x,t;\omega )}a( \theta _t\omega )
\end{align}
for $ t\in ( t_k,t_{k+1} )  $.

Let $d_\omega\ge \max \left\{ \frac{e^{-( \frac{\tilde{\mu}}{\mu}-1 ) \lVert A_{\omega} \rVert _{\infty}}}{\delta ( \frac{\tilde{\mu}}{\mu}-1 )},e^{( \frac{\tilde{\mu}}{\mu}-1 ) \lVert A_{\omega} \rVert _{\infty}} \right\}$. Then we have
$ d\delta ( \frac{\tilde{\mu}}{\mu}-1 ) e^{( \frac{\tilde{\mu}}{\mu}-1 ) A_{\omega}( t )}\ge 1$, $\forall d\ge d_\omega $. Note that if $ x\ge \int_0^t{c( s ;\omega ,\mu ) ds}+\frac{\ln d}{\tilde{\mu}-\mu}+\frac{A_{\omega}( t )}{\mu} $, then $\xi ( x,t;\omega ) =x-\int_0^t{c(s ;\omega ,\mu ) ds}\ge 0$ and $ \tilde{v}^{\mu ,d,A_{\omega}}( x,t ) \ge 0 $. Together with \eqref{inequality-A(t)}, we get that every term on the right hand side of  (\ref{more}) is less than or equal to zero.
\end{proof}

\medskip

\section{Random transition fronts}

In this section, we study the existence and stability of random transition fronts, and prove Theorem \ref{exist-thm} and \ref{stable-thm}.

\subsection{Existence of random transition fronts}

For any $\gamma>c_0$, let $0<\mu<\mu^*$ be such that $\frac{e^{\mu }+e^{-\mu}-2+\underline{a}}{\mu}=\gamma $. Then for every $\omega\in\Omega$, let  $c(t;\omega, \mu)=\frac{e^{\mu}+e^{-\mu}-2+a(\theta_t\omega)}{\mu}$ and  $ \hat{v}^{\mu}( x,t;\omega ) =e^{-\mu ( x-\int_0^t{c( s ;\omega ,\mu ) ds} )}$.
Then $ \hat{v}^{\mu}( x,t;\omega)$ satisfies
\begin{align*}
\partial &_t\hat{v}^{\mu}( x,t;\omega ) -H\hat{v}^{\mu}( x,t;\omega ) -a( \theta _t\omega ) \hat{v}^{\mu}( x,t;\omega )\\
=&\hat{v}^{\mu}( x,t;\omega ) [ \mu c( t;\omega ,\mu ) -( e^{\mu}+e^{-\mu}-2+a( \theta _t\omega ) ) ]=0,\ \ \text{for\ }x\in \mathbb{R},\ t\in \mathbb{R}.
\end{align*}
Then we have that
\begin{align*}
\partial _t\hat{v}^{\mu}( x,t;\omega )
&=H\hat{v}^{\mu}( x,t;\omega ) +a( \theta _t\omega ) \hat{v}^{\mu}( x,t;\omega )\\
&\ge H\hat{v}^{\mu}( x,t;\omega ) +a( \theta _t\omega ) \hat{v}^{\mu}( x,t;\omega ) ( 1-\hat{v}^{\mu}( x,t;\omega ) ) ,\ \ \text{for\ }x\in \mathbb{R},\ t\in \mathbb{R}.
\end{align*}
Hence, $ \hat{v}^{\mu}( x,t;\omega ) =e^{-\mu ( x-\int_0^t{c( s ;\omega ,\mu ) ds} )}$ is a super-solution of (\ref{main-eqn2}).
Denote
$$ \overline{v}^{\mu}( x,t;\omega ) =\min \{ 1,\hat{v}^{\mu}( x,t;\omega )\}. $$

\begin{lemma}\label{monotone-sup}
For $\omega\in\Omega_0$, we have that
$$ u( x,t-t_0;\overline{v}^{\mu}( \cdot ,t_0;\omega) ,\theta_{t_0}\omega) \le \overline{v}^{\mu}( x,t;\omega ) ,\ \ \forall x\in \mathbb{R},\ t\geq t_0,\  t_0\in\R. $$
\end{lemma}

\begin{proof}
For any constant $C$, $ \hat{u}( x,t;\omega) :=e^{Ct}\hat{v}^{\mu}( x,t;\omega)$ satisfies
\begin{align*}
\partial _t\hat{u}( x,t;\omega) &=( \partial _t\hat{v}^\mu( x,t;\omega) +C\hat{v}^{\mu}( x,t;\omega)) e^{Ct}\\
&\ge H\hat{u}( x,t;\omega) +C\hat{u}( x,t;\omega) +a( \theta _t\omega) \hat{u}( x,t;\omega)( 1-\hat{v}^{\mu}( x,t;\omega )),
\end{align*}
hence,
$$
\hat{u}( x,t;\omega) \ge \hat{u}( x,t_0;\omega) +\int_{t_0}^t{( H\hat{u}( x,\tau ;\omega) +C\hat{u}( x,\tau ;\omega) +a( \theta _{\tau}\omega) \hat{u}( x,\tau ;\omega) ( 1-\hat{v}^{\mu}( x,\tau ;\omega))) d\tau}.
$$
Denote $ \overline{u}( x,t;\omega) :=e^{Ct}\overline{v}^{\mu}( x,t;\omega)$. Then we also have
$$
\overline{u}( x,t;\omega) \ge \overline{u}( x,t_0;\omega) +\int_{t_0}^t{( H\overline{u}( x,\tau ;\omega ) +C\overline{u}( x,\tau ;\omega) +a( \theta _{\tau}\omega) \overline{u}( x,\tau ;\omega )( 1-\overline{v}^{\mu}( x,\tau ;\omega))) d\tau}
$$

Let  $ Q( x,t;\omega) =e^{Ct}( \overline{v}^{\mu}( x,t;\omega) -u( x,t-t_0;\overline{v}^{\mu}( \cdot ,t_0;\omega),\theta_{t_0}\omega))$. Then
$$ Q( x,t;\omega) \ge Q( x,t_0;\omega) +\int_{t_0}^t{( HQ( x,\tau ;\omega) +(C+b( x,\tau ;\omega)) Q( x,\tau ;\omega)) d\tau}, $$
where
$$ b( x,t ;\omega) =a( \theta _{t}\omega)( 1-\overline{v}^{\mu}( x,t ;\omega) -u( x,t-t_0 ;\overline{v}^{\mu}( \cdot ,t_0;\omega),\theta_{t_0}\omega)).$$
Choose $C>0$ such that $ b( x,t;\omega) -2+C>0 $ for all $t\ge t_0$, $x\in\mathbb{R}$ and a.e. $\omega\in\Omega$. By the arguments of Proposition \ref{comparison}, we have that
$$ Q( x,t;\omega) \ge Q( x,t_0;\omega) =0, $$
and hence for $\omega\in\Omega_0$, we have that
$$ u( x,t-t_0;\overline{v}^{\mu}( \cdot ,t_0;\omega),\theta_{t_0}\omega) \le \overline{v}^{\mu}( x,t;\omega) ,\,\,\,\,\forall x\in \mathbb{R},\,\,t\geq t_0,\ t_0\in\R. $$
\end{proof}

Next, we construct a sub-solution of (\ref{main-eqn2}).  Let $\tilde{\mu}>0$ be such that $\mu<\tilde{\mu}<\min\{2\mu,{\mu}^*\}$ and $\omega\in\Omega_0$. Let $A_\omega$ and $d_\omega$ be given by Lemma \ref{sub-solution-lemma}, and let
$$ x_{\omega}( t ) =\int_0^t{c( s;\omega ,\mu )}ds+\frac{{\ln}d_{\omega}+{\ln}\tilde{\mu}-{\ln}\mu}{\tilde{\mu}-\mu}+\frac{A_{\omega}( t )}{\mu}. $$
Recall that $$ \tilde{v}^{\mu ,d,A_{\omega}}( x,t, \omega) =e^{-\mu ( x-\int_0^t{c( s ;\omega ,\mu ) ds} )}-de^{( \frac{\tilde{\mu}}{\mu}-1 ) A_{\omega}( t ) -\tilde{\mu}( x-\int_0^t{c( s ;\omega ,\mu ) ds} )} $$
By calculation we have that for any given $t\in\mathbb{R}$,
\begin{equation}
\label{supremum}
 \tilde{v}^{\mu ,d_{\omega},A_{\omega}}( x_{\omega}( t ) ,t,\omega ) =\sup\limits_{x\in \mathbb{R}}\tilde{v}^{\mu ,d_{\omega},A_{\omega}}( x,t,\omega ) =e^{-\mu ( \frac{\ln {d_{\omega}}}{\tilde{\mu}-\mu}+\frac{A_{\omega}( t )}{\mu} )}e^{ -\mu \frac{ \ln{\tilde{\mu}}-  \ln{\mu}}{\tilde{\mu}-\mu}}( 1-\frac{\mu}{\tilde{\mu}} ) .
 \end{equation}
Define
$$ \underline{v}^{\mu}( x,t;\theta _{t_0}\omega ) =\left\{ \begin{array}{l} 	\tilde{v}^{\mu ,d_{\omega},A_{\omega}}( x,t+t_0,\omega ) ,\ \ \text{if\ }x\ge x_{\omega}( t+t_0 ) ,\\ 	\tilde{v}^{\mu ,d_{\omega},A_{\omega}}( x_{\omega}( t+t_0 ), t+t_0, \omega ) ,\ \ \text{if\ }x\le x_{\omega}( t+t_0 ) .\\ \end{array} \right.  $$
It is clear that
$$ 0<\underline{v}^{\mu}( x,t;\theta _{t_0}\omega ) <\overline{v}^{\mu}( x,t;\theta _{t_0}\omega ) \le 1,\ \ \forall t\in \mathbb{R},\ x\in \mathbb{R},\ t_0\in \mathbb{R}. $$
and
 \begin{align}
\label{super1}
\lim\limits_{x\rightarrow \infty} \sup\limits_{t\in\R, t_0\in \mathbb{R}} \frac{\underline{v}^{\mu}( x,t;\theta _{t_0}\omega )}{\overline{v}^{\mu}( x,t;\theta _{t_0}\omega )}=1.
\end{align}
Note that by the similar arguments as in Lemma \ref{monotone-sup}, we can prove that
$$ u( x,t-t_0;\underline{v}^{\mu}( \cdot ,t_0;\theta _{t_0}\omega ) ,\omega ) \ge \underline{v}^{\mu}( x,t;\omega )  $$
for $x\in\mathbb{R}$, $t\ge t_0$ and a.e. $\omega\in\Omega$.


Next, we prove Theorem \ref{exist-thm}.

\begin{proof}[Proof of Theorem \ref{exist-thm}]
By Lemma \ref{monotone-sup} we have that
$$ u( x,t-t_0;\overline{v}^{\mu}( \cdot ,t_0;\omega ) ,\theta _{t_0}\omega ) \le \overline{v}^{\mu}( x,t;\omega ) ,\ \ \forall x\in \mathbb{R},\ t\ge t_0,\ t_0\in \mathbb{R}. $$
It then follows that
$$ u( x,\tau_2-\tau_1;\overline{v}^{\mu}( \cdot ,-\tau_2;\omega) ,\theta _{-\tau_2}\omega) \le \overline{v}^{\mu}( x,-\tau_1;\omega) ,\ \ \forall x\in \mathbb{R},\ \ \tau_2>\tau_1.$$
Then we get that
$$u( x,t+\tau_1;u( \cdot,\tau_2-\tau_1;\overline{v}^{\mu}( \cdot ,-\tau_2;\omega) ,\theta _{-\tau_2}\omega),\theta _{-\tau_1}\omega)
\le u( x,t+\tau_1;\overline{v}^{\mu}( \cdot ,-\tau_1;\omega),\theta _{-\tau_1}\omega) $$
for $x\in \mathbb{R}$, $t\ge -\tau_1$, $\tau_2>\tau_1$,
and hence
$$u( x,t+\tau_2 ;\overline{v}^{\mu}( \cdot ,-\tau_2 ;\omega) ,\theta _{-\tau_2}\omega)\leq u( x,t+\tau_1 ;\overline{v}^{\mu}( \cdot ,-\tau_1 ;\omega) ,\theta _{-\tau_1}\omega) ,\,\forall x\in \mathbb{R}, t\ge -\tau_1,  \tau_2>\tau_1.$$
Therefore $\lim\limits_{\tau\rightarrow\infty} u( x,t+\tau ;\overline{v}^{\mu}( \cdot ,-\tau ;\omega) ,\theta _{-\tau}\omega)$ exists. Define
\begin{equation}
\label{sationary1}
 V(x,t;\omega)=\underset{\tau\rightarrow\infty}{\lim}u( x,t+\tau ;\overline{v}^{\mu}( \cdot ,-\tau ;\omega) ,\theta _{-\tau}\omega)
 \end{equation}
for $x\in \mathbb{R}$, $t\in\R$, $\omega \in \Omega _0$. Then $ V(x,t;\omega)$ is non-increasing in $x\in\R$ and by dominated convergence theorem we know that $ V(x,t;\omega)$ is a solution of \eqref{main-eqn2}.

We claim that, for every $\omega\in\Omega_0$,
\begin{align}
\label{uniformly}
\underset{x\rightarrow -\infty}{\lim}V( x+\int_0^t{c(s;\omega ,\mu ) ds},t;\omega ) =1\ \ \text{uniformly\ in}\ t\in \mathbb{R}.
\end{align}
In fact, fix any $\omega\in\Omega_0$, let
$ \hat{x}_\omega=\frac{\ln d_{\omega}+\ln{\tilde{\mu}}-\ln\mu}{\tilde{\mu}-\mu}-\frac{\lVert A_{\omega} \rVert _{\infty}}{\mu},$
it follows from $ \underline{v}^{\mu}( x,t;\omega) \le V( x, t; \omega) $ and \eqref{supremum} that
$$ 0<( 1-\frac{\mu}{\tilde{\mu}}) e^{-\mu(\frac{\ln d_{\omega}+\ln{\tilde{\mu}}-\ln\mu}{\tilde{\mu}-\mu}+\frac{\lVert A_{\omega} \rVert _{\infty}}{\mu})}\le \inf\limits_{t\in \mathbb{R}}V(\hat{x}_\omega+\int_0^t{c( s ;\omega ,\mu) ds, t; \omega} ).  $$
Let $ u_0( x) \equiv u_0:=\inf\limits_{t\in \mathbb{R}}V(\hat{x}_\omega+\int_0^t{c(s ;\omega ,\mu) ds, t; \omega})$, and $\tilde{u}_0(x)$ be uniformly continuous such that $\tilde{u}_0(x)=u_0(x)$ for $x<\hat{x}_\omega-1$ and $\tilde{u}_0(x)=0$ for $x\geq \hat{x}_\omega$.
Then $\underset{n\rightarrow \infty}{\lim}\tilde{u}_0( x-n) =u_0( x)$ locally uniformly in $x\in \mathbb{R}.$
Note that by the proof of Proposition \ref{stable-lemma}, we have
$$ \underset{t\rightarrow \infty}{\lim}u( x,t;u_0,\theta _{t_0}\omega) =1$$
uniformly in $t_0\in \mathbb{R}$ and $x\in \mathbb{R}$.
Then for any $\epsilon>0$, there is $T:=T(\epsilon)>0$ such that $$ 1>u( x,T;u_0,\theta _{t_0}\omega) >1-\epsilon, \ \ \forall t_0\in \mathbb{R},\ x\in \mathbb{R}. $$
Therefore, by (H) and the definition of $c(t,\omega,\mu)$ we derive, $$ 1>u( x+\int_0^T{c( s ;\theta_{t_0}\omega ,\mu) ds},T;u_0,\theta _{t_0}\omega) >1-\epsilon, \,\,\,\,\forall t_0\in \mathbb{R},\,\,x\in \mathbb{R}. $$
By Proposition \ref{convergence-lemma}, there is $N:=N(\epsilon)>1$ such that
$$ 1>u( \int_0^T{c( s ;\theta_{t_0}\omega ,\mu ) ds },T;\tilde{u}_0(\cdot-N),\theta _{t_0}\omega) >1-2\varepsilon, \,\,\,\,\forall t_0\in \mathbb{R}. $$
That is,
$$ 1>u( \int_0^T{c( s ;\theta_{t_0}\omega ,\mu ) ds}-N,T;\tilde{u}_0(\cdot),\theta _{t_0}\omega) >1-2\varepsilon, \,\,\,\,\forall t_0\in \mathbb{R}. $$
Note that
$$ V( x+\int_0^{t-T}{c( s;\omega ,\mu) ds},t-T;\omega ) \ge \tilde{u}_0( x), \ \ \forall t\in \mathbb{R},\ x\in \mathbb{R}. $$
and
$$ \int_0^t{c( s ;\omega ,\mu) ds}=\int_0^T{c(s ;\theta _{t-T}\omega ,\mu) ds}+\int_0^{t-T}{c( s ;\omega ,\mu ) d s}. $$
Then we get
\begin{align*}
\label{uniformly1}
 1&>V( x+\int_0^t{c( s ;\omega ,\mu ) ds},t;\omega)
 \\&=u( x+\int_0^T{c( s ;\theta_{t-T}\omega ,\mu ) ds},T; V( \cdot+\int_0^{t-T}{c( s;\omega ,\mu ) ds},t-T;\omega), \theta_{t-T}\omega ) \\&>1-2\epsilon, \ \ \ \ \ \ \forall t\in \mathbb{R},\ x\le -N.
\end{align*}
Thus \eqref{uniformly} follows.

Note that by  \eqref{super1} we have that, for every $\omega\in\Omega_0$,
$$ \lim\limits_{x\rightarrow \infty}\sup\limits_{t\in \mathbb{R}} \frac{V( x+\int_0^t{c( s ;\omega ,\mu ) ds},t;\omega ) }{e^{-\mu x}}=1.$$

Set
$$\tilde{ \varPhi} (x, t;\omega) =V( x+\int_0^t{c( s;\omega ,\mu ) ds,t;\omega} ) \,\,\,\text{ and }\,\,\,\varPhi(x,\omega)=\tilde{ \varPhi} (x, 0;\omega) . $$
We claim that $\tilde{ \varPhi} (x, t;\omega)$ is stationary
ergodic in $t$, that is,  for a.e. $\omega\in\Omega$,
\begin{equation*}
\label{stationary ergodic}
\tilde{ \varPhi} ( x,t;\omega ) =\tilde{\varPhi}( x,0;\theta _t\omega ) .
\end{equation*}
In fact, note that for $\omega\in\Omega$,
\begin{align}
\label{speed1}
\int_{-\tau}^t{c( s;\omega ,\mu )}ds
&=\int_{-\tau}^t{\frac{e^{\mu}+e^{-\mu}-2+a( \theta _s\omega )}{\mu}ds} \nonumber \\
&=\frac{e^{\mu}+e^{-\mu}-2}{\mu}( t+\tau ) +\int_{-\tau}^t{\frac{a( \theta _s\omega )}{\mu}ds}
\end{align}
and
\begin{align}
\label{speed2}
\int_{-( t+\tau )}^0{c( s;\theta _t\omega ,\mu ) ds}
&=\int_{-( t+\tau )}^0{\frac{e^{\mu}+e^{-\mu}-2+a( \theta _s\circ\theta _t\omega )}{\mu}ds} \nonumber \\
&=\frac{e^{\mu}+e^{-\mu}-2}{\mu}( t+\tau ) +\int_{-( t+\tau )}^0{\frac{a( \theta _{s+t}\omega )}{\mu}ds} \nonumber \\
&=\frac{e^{\mu}+e^{-\mu}-2}{\mu}( t+\tau ) +\int_{-\tau}^t{\frac{a( \theta _s\omega )}{\mu}ds}.
\end{align}
Combining \eqref{speed1} with \eqref{speed2}, we derive $ \int_{-\tau}^t{c( s;\omega ,\mu ) ds}=\int_{-( t+\tau )}^0{c( s;\theta _t\omega ,\mu ) ds}$ for  $\tau \ge 0$ and $ t\in \mathbb{R} $.
Recall that $$ \overline{v}^{\mu}( x,t;\omega ) =\min \left\{ 1,e^{-\mu ( x-\int_0^t{c( s ;\omega ,\mu ) ds})} \right\}. $$
Then we have
\begin{align*}
\tilde{ \varPhi} ( x,t;\omega )=
&\lim\limits_{\tau \rightarrow \infty}u( x+\int_0^t{c( s;\omega ,\mu ) ds},t+\tau ;\overline{v}^{\mu}( \cdot ,-\tau ;\omega ) ,\theta _{-\tau}\omega )\\
=&\lim\limits_{\tau \rightarrow \infty}u( x,t+\tau ;\overline{v}^{\mu}( \cdot+\int_0^t{c( s;\omega ,\mu ) ds} ,-\tau ;\omega ) ,\theta _{-\tau}\omega )\\
=&\lim\limits_{\tau \rightarrow \infty}u( x,t+\tau ;\overline{v}^{\mu}( \cdot ,-( t+\tau ) ;\theta _t\omega ) ,\theta _{-\tau}\omega )\\
=&\lim\limits_{\tau \rightarrow \infty}u( x,t+\tau ;\overline{v}^{\mu}( \cdot ,-( t+\tau ) ;\theta _t\omega ) ,\theta _{t-(t+\tau)}\omega )\\
=&\lim\limits_{\tau \rightarrow \infty}u( x,\tau ;\overline{v}^{\mu}( \cdot ,-\tau ;\theta _t\omega ) ,\theta _{t-\tau}\omega )\\
=&\tilde{\varPhi} ( x,0;\theta _t\omega ).
\end{align*}
The claim  thus follows and we get the desired random profile $\varPhi(x,\omega)$.

\end{proof}

\subsection{Stability of random transition fronts}
\begin{proof}[Proof of Theorem \ref{stable-thm}]

For any $\omega \in \Omega _0 $ and given $\mu\in(0,\mu^*)$,
$u(t;\omega)=\{u_i(t;\omega)\}_{i\in\mathbb{Z}}$ with $u_i( t;w ) ={\varPhi} ( i-\int_0^t{c( s;\omega,\mu ) d s },\theta _t\omega )$ is a random transition front of \eqref{main-eqn}. Let $ u^0\in l^{\infty}( \mathbb{Z} )$, $\ u^0=\{ u_{i}^{0}\} _{i\in \mathbb{Z}}$ satisfy
\begin{equation*}
\underset{i\le i_0}{\inf}u_{i}^{0}>0\ \ \ \forall i_0\in \mathbb{Z},\ \ \underset{i\rightarrow \infty}{\lim}\frac{u_{i}^{0}}{u_i(0;\omega)}=1.
\end{equation*}
Then there is $ \alpha \ge 1$ such that
$$ \frac{1}{\alpha}\le \frac{u_{i}^{0}}{u_i(0;\omega)}\le \alpha ,\ \ \forall i\in \mathbb{Z}. $$
By comparison principle we get that
\begin{align*}
 u_i( t;u^0,\omega ) \le u_i( t;\alpha u_\cdot(0;\omega) ,\omega ) ,\ \ \forall i\in \mathbb{Z},\  t\ge 0
\end{align*}
and
\begin{align}
\label{stable2}
 u_i( t, \omega ) \le u_i( t;\alpha u^0,\omega ), \ \ \forall i\in \mathbb{Z},\  t\ge 0.
\end{align}
Also, we have
$$ \frac{d}{dt}( \alpha u_i(t;u^0,\omega) ) \ge H( \alpha u_i(t;u^0,\omega) ) +a( \theta _t\omega ) \alpha u_i(t;u^0,\omega) ( 1-\alpha u_i(t;u^0,\omega) ).  $$
Again by comparison principle and \eqref{stable2} we have that
\begin{equation*}
\label{stable3}
u_i(t,\omega)\le u_i(t;\alpha u^0,\omega ) \le \alpha u_i(t;u^0,\omega ) ,\ \ \forall i\in\Z, \  t\ge 0.
\end{equation*}
Similarly, we can get
$$ u_i(t;u^0,\omega ) \le \alpha u_i( t,\omega ) ,\ \ \forall i\in\Z, \ t\ge 0. $$
Thus for every $ t\ge 0 $,we can define $\alpha ( t ) \ge 1$ as
\begin{align}
\label{stable4}
\alpha ( t ) =\inf\{ \alpha \ge 1 \ \ | \ \ \frac{1}{\alpha}\le \frac{u_i(t;u^0,\omega )}{u_i(t,\omega )}\le \alpha \ \text{for\ any\ }i\in \mathbb{Z} \}.
\end{align}
It's easy to see that $ \alpha ( t_2 ) \le \alpha ( t_1 )  $ for every $ 0\le t_1 \le t_2 $. Therefore
$$ \alpha _{\infty}:=\inf \{ \alpha ( t ) \ \ |\ \  t\ge 0 \} =\underset{t\rightarrow \infty}{\lim}\alpha ( t )  $$ exists.
Then to get Theorem \ref{stable-thm}, it is sufficient to  prove $ \alpha _{\infty}=1 $.

Suppose by contradiction that $\alpha_{\infty}>1$. Let $ 1<\alpha <\alpha _{\infty} $ be fixed, we first prove that there is $I_{\alpha}\gg1$ such that
\begin{align}
\label{stable5}
\frac{1}{\alpha}\le \frac{u_i(t;u^0,\omega )}{u_i( t, \omega )}\le \alpha ,\ \ \forall i\ge I_{\alpha}+\int_0^t{c( s;\omega,\mu ) ds},\ t\ge 0.
\end{align}
To this end, we only need to prove that
\begin{align}
\label{stable20}
 \underset{i\rightarrow \infty}{\lim}\frac{u_i(t;u^0,\omega )}{e^{-\mu ( i-\int_0^t{c( s;\omega ,\mu ) ds} )}}=1\ \ \text{uniformly in } t\geq0.
\end{align}

In fact, since for every $\epsilon >0$, there is $J_{\epsilon,\omega}\gg1$ such that
$$ 1-\epsilon \le \frac{u_{i}^{0}}{u_i(0,\omega)}\le 1+\epsilon, \ \ \forall i\ge J_ {\epsilon,\omega}. $$
Let $A_\omega (t)$ be as in Lemma \ref{sub-solution-lemma}.
Since
$$ e^{-\mu ( i-\int_0^t{c( s;\omega ,\mu ) ds} )}-d_{\omega}e^{A_{\omega}( t ) -\tilde{\mu} ( i-\int_0^t{c( s;\omega ,\mu ) ds} )}\le u_i(t,\omega ) \le e^{-\mu ( i-\int_0^t{c( s;\omega ,\mu ) ds} )}, $$
then
\begin{align}
\label{stable13}
( 1-\epsilon ) e^{-\mu i}-( 1-\epsilon ) d_{\omega}e^{A_{\omega}( 0 ) -\tilde{\mu} i}\le u_{i}^{0}\le ( 1+\epsilon ) e^{-\mu i},\ \ \forall i\ge J_{\epsilon,\omega}.
\end{align}
We claim that there is $d\gg1$ such that
\begin{align}
\label{stable14}
 ( 1-\epsilon ) e^{-\mu i}-de^{A_{\omega}( 0 ) -\tilde{\mu} i}\le u_{i}^{0}\le ( 1+\epsilon ) e^{-\mu i}+de^{A_{\omega}( 0 ) -\widetilde{\mu} i},\,\,\,\,\forall i\in \mathbb{Z}.
\end{align}
Indeed, note that
$$ \lVert u^{0} \rVert _{\infty}e^{\tilde{\mu} J_{\epsilon,\omega}+| A_{\omega}( 0 ) |}e^{A_{\omega}( 0 ) -\tilde{\mu} i}\ge \lVert u^{0} \rVert _{\infty}e^{(\tilde{\mu}-\tilde{\mu}) J_{\epsilon,\omega}}\ge u_{i}^{0},\ \ \forall i\le J_{\epsilon,\omega}. $$
Hence
\begin{align*}
u_{i}^{0}\le d_{\epsilon ,\omega}e^{A_{\omega}( 0 ) -\tilde{\mu }i}\le ( 1+\epsilon ) e^{-\mu i}+d_{\epsilon ,\omega}e^{A_{\omega}( 0 ) -\tilde{\mu } i},\ \ \forall i\le J_{\epsilon,\omega},
\end{align*}
where $ d_{\epsilon ,\omega}=:\lVert u^{0} \rVert _{\infty}e^{\tilde{\mu} J_{\epsilon,\omega}+| A_{\omega}( 0 ) |} $. Combining this with \eqref{stable13}, we obtain that
\begin{align}
\label{stable15}
u_{i}^{0}\le ( 1+\epsilon ) e^{-\mu i}+d_{\epsilon,\omega}e^{A_{\omega}( 0 ) -\widetilde{\mu} i},\,\,\,\,\forall i\in \mathbb{Z}.
\end{align}
On the other hand, for every $d>1$, the function $ \mathbb{Z}\owns i\mapsto ( 1-\epsilon ) e^{-\mu i}-de^{A_{\omega}( 0 ) -\tilde{\mu} i} $ attains its maximum value at $ J_d:=[\frac{\text{In}{ ( \frac{d\tilde{\mu} e^{A_{\omega}(0)}}{(1-\epsilon) \mu} )}}{\tilde{\mu }-\mu}] $ or $[\frac{\text{In}{ ( \frac{d\tilde{\mu} e^{A_{\omega}(0)}}{(1-\epsilon) \mu} )}}{\tilde{\mu }-\mu}]  +1. $
Note that $ \underset{d\rightarrow \infty}{\lim}J_d=\infty  $ and  $$ \underset{d\rightarrow \infty}{\lim}( ( 1-\epsilon ) e^{-\mu J_d}-de^{A_{\omega}( 0 ) -\tilde{\mu} J_d} ) =0. $$
Then there is $ \tilde{d}_{\epsilon ,\omega}\gg ( 1-\varepsilon ) d_{\omega} $ such that $ J_{\tilde{d}_{\epsilon ,\omega}}\ge J_{\epsilon,\omega} $ and
$$ ( 1-\epsilon ) e^{-\mu J_{\tilde{d}_{\epsilon ,\omega}}}-\tilde{d}_{\epsilon ,\omega}e^{A_{\omega}( 0 ) -\tilde{\mu} J_{\tilde{d}_{\epsilon ,\omega}}}\le \underset{i\le J_{\epsilon,\omega}}{\inf}u_{i}^{0}. $$
Together with \eqref{stable13}, it follows that
\begin{align}
\label{stable16}
 ( 1-\epsilon ) e^{-\mu i}-de^{A_{\omega}( 0 ) -\tilde{\mu} i}\le u_{i}^{0},\,\,\,\,\forall i\in \mathbb{Z},\ \ d\geq \tilde{d}_{\epsilon,\omega} .
\end{align}
By \eqref{stable15} and \eqref{stable16} we drive that the claim \eqref{stable14} holds for every $ d\ge \max \{ \tilde{d}_{\epsilon ,\omega},d_{\epsilon ,\omega} \}$. Thus by similar arguments as proving Lemma \ref{sub-solution-lemma}, we can get that for $d\gg1$,
$$ \dot{\tilde{u}}_i(t,\omega ) \le H\tilde{u}_i(t,\omega ) +a( \theta _t\omega ) \tilde{u}_i(t,\omega ) ( 1-\tilde{u}_i(t,\omega ) )$$
on the set $ D_{\epsilon}:=\{ ( i,t ) \in \mathbb{Z}\times \mathbb{R}^+| \tilde{u}_i(t,\omega )\ge 0 \}  $, where $ \tilde{u}_i(t,\omega ) =( 1-\epsilon ) e^{-\mu (i-\int_0^t{c( s;\omega ,\mu ) ds})}-de^{A_{\omega}( t ) -\tilde{\mu} ( i-\int_0^t{c( s;\omega ,\mu ) ds})} $. Then by comparison principle we get that
\begin{equation*}
\label{stable17}
( 1-\epsilon ) e^{-\mu ( i-\int_0^t{c( s;\omega ,\mu ) ds} )}-de^{A_{\omega}( t ) -\tilde{\mu }( i-\int_0^t{c( s;\omega ,\mu ) ds} )}\le u_i(t;u^0,\omega )\end{equation*}  for $i\in \mathbb{Z}$, $t\ge 0$, $d\gg 1$.
Similarly, we can get that
$$u_i(t;u^0,\omega ) \leq ( 1+\epsilon ) e^{-\mu ( i-\int_0^t{c( s;\omega ,\mu ) ds})}+de^{A_{\omega}( t) -\tilde{\mu} ( i-\int_0^t{c( s;\omega ,\mu ) ds})},\ \ \forall i\in \mathbb{Z}, \  t\ge 0,\ d\gg 1.$$
 \eqref{stable20} and \eqref{stable5} then follow form the last two inequalities and the arbitrariness of $\epsilon>0$.

Next, let  $I_{\alpha}$ be given by \eqref{stable5} and set
$$ m_{\alpha}:=\frac{1}{\alpha _0}\inf\limits_{t\ge 0,i-\int_0^t{c( s ;\omega ,\mu ) ds}\le I_{\alpha}} u_i( t, \omega ),  $$
where $ \alpha _0=\alpha ( 0 ) =\sup\limits_{t\ge 0}\alpha ( t ) . $ It then follows from \eqref{stable4} that
$$ m_{\alpha}\le \min \{ u_i(t, \omega ) ,u_i(t;u^0,\omega ) \} ,\ \ \forall i\le I_{\alpha}+\int_0^t{c( s ;\omega ,\mu ) ds},\  t\ge 0. $$
By (H) there is $ T=T( \omega ) \ge 1 $ such that
\begin{align}
\label{stable6}
0<\frac{\underline{a}T}{2}<\int_\tau^{\tau+T}{a( \theta _{s}\omega )}ds <2\overline{a}T<\infty ,\ \ \forall \tau\in \mathbb{R}.
\end{align}
Let $ 0<\delta \ll 1 $ satisfy
\begin{align}
\label{stable7}
\alpha <e^{-2\delta T\overline{a}}\alpha _{\infty}\ \ \ \text{\and} \ \ ( ( \alpha _{\infty}-1 ) -\alpha _0( 1-e^{-2\delta T\overline{a}} ))m_{\alpha}>\delta.
\end{align}

We claim that
\begin{align}
\label{stable8}
\alpha ( ( k+1 ) T ) \le e^{-\delta \int_{kT}^{( k+1 ) T}{a( \theta _s\omega) ds}}\alpha ( kT ) ,\ \ \forall k\ge 0.
\end{align}
In fact, Set $ W_k( i,t;\omega ) =e^{\delta \int_{kT}^{t+kT}{a( \theta _s\omega ) ds}}u_i(t+kT;u^0,\omega )$, $ V_k( i,t;\omega ) =u_i(t; u_{\cdot}(0; \theta_{kT}\omega) ,\theta _{kT}\omega)  $, $ a_k( t ) =a( \theta _{t+kT}\omega )$ and $ \alpha _k=\alpha ( kT ) $,
it follows from \eqref{stable6} that
\begin{align}
\label{stable9}
\frac{d}{dt}W_k& =\delta a_k( t ) W_k+HW_k+a_k( t ) W_k( 1-u_i(t+kT;u^0,\omega ) ) \nonumber \\
& =HW_k+a_k( t ) W_k( 1-W_k ) +a_k( t ) W_k( ( 1-e^{-\delta \int_{kT}^{t+kT}{a( \theta _s\omega ) ds}} ) W_k+\delta ) \nonumber \\
& \le HW_k+a_k( t ) W_k( 1-W_k ) +a_k( t ) W_k( ( 1-e^{-2\delta T\overline{a}} ) W_k+\delta )
\end{align}
for every $ t\in ( 0,T )$, $ i\in \mathbb{Z}$ and $ k\ge 0 $. Also, it follows from \eqref{stable7} and $\alpha_\infty \leq \alpha_k \leq \alpha_0$ that
\begin{align}
\label{stable10}
 \frac{d}{dt}&( \alpha _kV_k )
-H( \alpha _kV_k )\nonumber\\=&a_k( t ) ( \alpha _kV_k ) ( 1-V_k ) \nonumber \\
=&a_k( t ) ( \alpha _kV_k ) ( 1- \alpha _kV_k ) +a_k( t) ( \alpha _kV_k ) ( ( 1-e^{-2\delta T\bar{a}} ) ( \alpha _kV_k ) +\delta ) \nonumber \\
& +a_k( t ) ( \alpha _kV_k ) ( ( ( \alpha _k-1 ) -( 1-e^{-2\delta T\bar{a}} ) \alpha _k ) V_k-\delta ) \nonumber \\
 \ge& a_k( t ) ( \alpha _kV_k ) ( 1- \alpha _kV_k ) +a_k( t ) ( \alpha _kV_k ) ( ( 1-e^{-2\delta T\bar{a}} ) ( \alpha _kV_k ) +\delta ) \nonumber \\
& +a_k( t ) ( \alpha _kV_k ) (( ( \alpha _{\infty}-1 ) -( 1-e^{-2\delta T\bar{a}} ) \alpha _0 ) m_{\alpha}-\delta ) \nonumber \\
 \ge& a_k( t ) ( \alpha _kV_k ) ( 1- \alpha _kV_k ) +a_k( t ) ( \alpha _kV_k) ( ( 1-e^{-2\delta T\bar{a}} ) ( \alpha _kV_k ) +\delta )
\end{align}
for $ i\le I_{\alpha}+\int_0^{t+kT}{c( s ;\omega ,\mu ) ds}$, $ 0\le t\le T$ and $ k\ge 0. $ Therefore, it follows from \eqref{stable4}, \eqref{stable5},
$ e^{\delta \int_{kT}^{( k+1 ) T}{a( \theta _s\omega ) ds}}\alpha \le \alpha _{\infty}\le \alpha _k $, and comparison principle that
$$ e^{\delta \int_{kT}^{t+kT}{a( \theta _s\omega ) ds}}u_i(t+kT;u^0,\omega) \le \alpha _ku_i(t+kT, \omega ) $$
for $i\le I_{\alpha}+\int_0^{t+kT}{c( s ;\omega ,\mu ) ds}$, $ t\in [ 0,T ] $ and  $k\ge 0$.
That is
$$ u_i(t+kT;u^0,\omega ) \le e^{-\delta \int_{kT}^{t+kT}{a( \theta _s\omega ) ds}}\alpha _ku_i(t+kT, \omega )  $$
for $i\le I_{\alpha}+\int_0^{t+kT}{c( s ;\omega ,\mu ) ds}$, $ t\in [ 0,T ]$ and $ k\ge 0. $
Note that $ \alpha \le e^{-\delta \int_{kT}^{( k+1 ) T}{a( \theta _s\omega ) ds}}\alpha _{\infty}\le e^{-\delta \int_{kT}^{( k+1 ) T}{a( \theta _s\omega) ds}}\alpha _k, $ then by \eqref{stable5} we have that
$$ u_i(t+kT;u^0,\omega ) \le e^{-\delta \int_{kT}^{t+kT}{a( \theta _s\omega ) ds}}\alpha _ku_i( t+kT, \omega )$$
for $ i\ge I_{\alpha}+\int_0^{t+kT}{c( s ;\omega ,\mu ) ds}$, $ t\in [ 0,T ]$ and $ k\ge 0. $
Thus,
\begin{align}
\label{stable11}
u_i(t+kT;u^0,\omega ) \le e^{-\delta \int_{kT}^{t+kT}{a( \theta _s\omega ) ds}}\alpha _ku_i(t+kT,\omega )
\end{align}
for  $i\in \mathbb{Z}$, $t\in [ 0,T ]$ and $k\ge0$.
By interchanging $W_k$ and $V_k$ in \eqref{stable9} and \eqref{stable10}, we can also get that
\begin{align}
\label{stable12}
u_i( t+kT,\omega ) \le e^{-\delta \int_{kT}^{t+kT}{a( \theta _s\omega ) ds}}\alpha _ku_i(t+kT;u^0,\omega )
\end{align}
for  $i\in \mathbb{Z}$, $t\in [ 0,T ]$ and $k\ge0$.
Then the claim \eqref{stable8} follows from \eqref{stable11} and \eqref{stable12}.

Therefore, it follows from \eqref{stable8} that,
\begin{equation}\label{stability00}
\alpha _{\infty}\le \alpha ( ( k+1 ) T ) \le e^{-\delta \sum_{i=0}^k{\int_{iT}^{( i+1 ) T}{a( \theta _s\omega ) ds}}}\alpha ( 0 ) =e^{-\delta \int_0^{( k+1 ) T}{a( \theta _s\omega ) ds}}\alpha _0\end{equation}
for any $k\ge 0$.
Note that $ \int_0^{\infty}{a( \theta _s\omega ) ds}=\infty$ for $\omega\in\Omega_0$. Then by letting $k\rightarrow\infty$ in \eqref{stability00}, we get that $\alpha_\infty \leq0$, a contradiction. So we get that $\alpha_\infty=1$, which leads to the asymptotic stability of the random transition fronts.
\end{proof}

\smallskip


\end{document}